\documentclass[12pt]{article}

\usepackage{geometry}
\usepackage{graphics}
\usepackage{amssymb}
\usepackage{amsbsy}
\usepackage{amsmath}
\usepackage{amsthm}
\usepackage{amsfonts}
\usepackage{extarrows}
\usepackage{dsfont}
\usepackage{float}
\usepackage[enableskew]{youngtab}
\usepackage{ytableau}
\usepackage[all]{xy}
\usepackage{setspace}
\usepackage{enumitem}

\usepackage{microtype}

\usepackage[square]{natbib}

\newtheorem{thm}{Theorem}[section]
\newtheorem{lem}[thm]{Lemma}
\newtheorem{prop}[thm]{Proposition}
\newtheorem{cor}[thm]{Corollary}

\newtheorem{remark}[thm]{Remark}
\newtheorem{dfn}[thm]{Definition}

\newtheorem{conjecture}[thm]{Conjecture}
\newtheorem*{thm*}{Theorem}

\definecolor{grey}{gray}{0.41}

\newcommand{\mc}{\mathcal{C}}

\newcommand{\md}{\mathcal{D}}

\newcommand{\none}{n-1}

\newcommand{\qs}{\mathcal{S}}

\DeclareMathOperator{\rw}{rw}
\DeclareMathOperator{\cw}{cw}
\DeclareMathOperator{\bw}{bw}

\DeclareMathOperator{\SYam}{SYam}

\DeclareMathOperator{\SYT}{SYT}
\DeclareMathOperator{\SSYT}{SSYT}
\DeclareMathOperator{\SRCT}{SRCT}
\DeclareMathOperator{\SRT}{SRT}

\DeclareMathOperator{\slink}{slink}
\DeclareMathOperator{\fl}{fl}
\DeclareMathOperator{\rev}{rev}
\DeclareMathOperator{\flip}{flip}
\DeclareMathOperator{\ID}{ID}

\numberwithin{equation}{section}

\begin{document}

\title{From symmetric fundamental expansions\\ to Schur positivity}           

\author{Austin Roberts
\\ 
Department of Mathematics\\ Highline College\\
 Des Moines, WA 98198, USA\\
 }

\date{\today}

\maketitle


\begin{abstract}     
We consider families of quasisymmetric functions with the property that if a symmetric function $f$ is a positive sum of functions in one of these families, then $f$ is necessarily a positive sum of Schur functions.  Furthermore, in each of the families studied, we give a combinatorial description of the Schur coefficients of $f$. We organize six such families into a poset, where functions in higher families in the poset are always positive integer sums of functions in each of the lower families. This poset includes the Schur functions, the quasisymmetric Schur functions, the fundamental quasisymmetric generating functions of shifted dual equivalence classes, as well as three new families of functions --- one of which is conjectured to be a basis of the vector space of quasisymmetric functions. Each of the six families is realized as the fundamental quasisymmetric generating functions over the classes of some refinement of dual Knuth equivalence. Thus, we also produce a poset of refinements of dual Knuth equivalence. In doing so, we define quasi-dual equivalence to provide classes that generate quasisymmetric Schur functions.
\end{abstract}

\pagebreak

\section{Introduction}

The problem of how to express a symmetric function in terms of the basis of Schur functions arises prominently in many fields, including algebraic combinatorics, representation theory, and statistical mechanics, amongst others. For instance, showing that a function is a positive integer sum of Schur functions (Schur positive) is equivalent to showing that the function corresponds to a representation of the general linear group, where the coefficients of said sum give the multiplicities of irreducible sub representations. See \citep{sagan} or \citep{Stanley2}) for a treatment. In many cases, such as Macdonald polynomials or plethysms of Schur functions, a symmetric function has a known expansion in terms of the fundamental quasisymmetric functions while an explicit expansion over the Schur functions remains elusive (see \citep{HHL} and \citep{LW}).

In this paper, we consider families of quasisymmetric functions with the property that if a symmetric function $f$ is a positive sum of functions in one of these families, then $f$ is necessarily a positive sum of Schur functions. Furthermore, in each of the families studied, we give a combinatorial description of the Schur coefficients of $f$. We organize six such families into a poset, where functions in higher families in the poset are positive integer sums of functions in each of the lower families. This poset includes the Schur functions, quasisymmetric Schur functions, the fundamental quasisymmetric generating functions of shifted dual equivalence classes $\{f^{(h)}\}$, as well as three new families of functions $\{f^{(k)}\}$ for $k=0,1,2$ --- conjecturing that $\{f^{(2)}\}$ forms a basis for the vector space of quasisymmetric functions. 
 
The poset of functions is first realized using a poset of equivalence relations. Each of the families is defined as sums of fundamental quasisymmetric functions over equivalence classes  on standard Young tableaux of fixed partition shape $\lambda$, $\SYT(\lambda)$. A higher position in our poset represents a courser relation. We then use the Robinson-Schensted-Knuth (RSK) correspondence to turn equivalence classes on tableaux into equivalence classes on permutations, realizing each equivalence relation as some restriction of dual Knuth equivalence.
An illustration of said poset of equivalence classes,  the related poset of quasisymmetric functions, and the generators of these classes  can be found in Figure~\ref{equiv poset}.

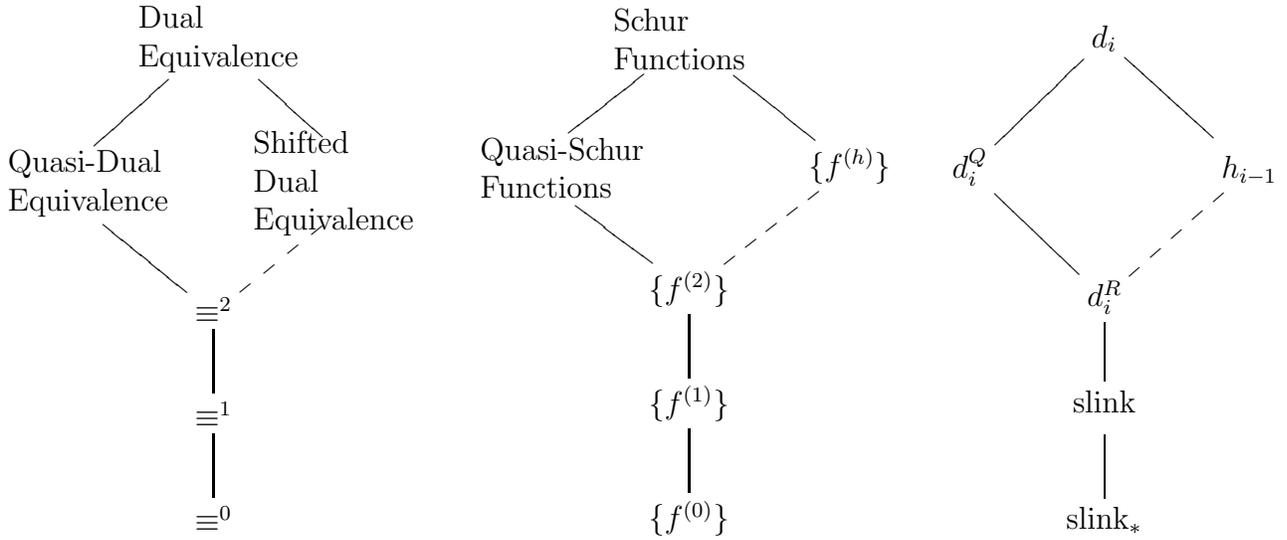
\begin{figure}[h]
 \[\begin{array}{ccc}
  \xymatrix{ 
  &\hspace{-.2in}\parbox{2cm}{Dual\\ Equivalence}\hspace{-.2in} \ar@{-}[dl]\ar@{-}[dr]&\\
 \parbox{2.2cm}{Quasi-Dual\\ Equivalence}\hspace{-.4in}\ar@{-}[dr]& &\hspace{-.4in}\parbox{2.2cm}{Shifted Dual\\ Equivalence}\ar@{--}[dl]\\
  &\equiv^2\ar@{-}[d]&\\
  &\equiv^1\ar@{-}[d]&\\
  &\equiv^0&
  }
    &  \hspace{.1in}
  \xymatrix{
   &\hspace{-.2in}\parbox{2cm}{Schur\\ Functions}\hspace{-.2in} \ar@{-}[dl]\ar@{-}[dr]&\\
 \parbox{2.2cm}{Quasi-Schur Functions}\hspace{-.4in}\ar@{-}[dr]& &\{f^{(h)}\}\ar@{--}[dl]\\
  &\{f^{(2)}\}\ar@{-}[d]&\\
  &\{f^{(1)}\}\ar@{-}[d]&\\
  &\{f^{(0)}\}&
  }
    &  \hspace{.1in}
   \xymatrix{ 
 &\raisebox{0cm}[.6cm]{$ d_i $}\ar@{-}[dl]\ar@{-}[dr]&\\
   \raisebox{0cm}[.6cm]{$d^Q_i$} \ar@{-}[dr]& & \raisebox{0cm}[.6cm]{$h_{i-1}$}\ar@{--}[dl]\\
  &\raisebox{0cm}[.5cm]{$ d_i^R $}\ar@{-}[d]&\\
  &\raisebox{.2cm}[.5cm]{$\slink$} \ar@{-}[d]&\\
  &\raisebox{.2cm}[.5cm]{$\slink_*$} &
  }
  \end{array}\]
\caption{At left, the poset of equivalence classes. At center, the associated quasisymmetric functions of the equivalence class. At right, the generators of the equivalence relations. The dashed lines denotes that a transformation is required, reversing all permutations in each equivalence class for the left and right posets.}\label{equiv poset}
\end{figure}

While detailed definitions will be given later, we can describe the main results of the paper as follows. %
In \citep{Gessel}, Gessel used the set $\SYT(\lambda)$ to express the Schur function  $s_\lambda$ as,

\begin{align}\label{Schur def}
s_\lambda =\sum_{T\in\SYT(\lambda)} F_{\ID(T)}(X),
\end{align}

\noindent where $F_{\ID(T)}(X)$ is the fundamental quasisymmetric function

\begin{align}
F_{\ID(T)}(X)=\sum_{j_1\leq \ldots \leq j_n \atop j_k = j_{k+1} \implies j\notin \ID(T)} x_{j_1}\cdots x_{j_n}.
\end{align}

\noindent
Here, $\ID(T)$ is the inverse descent set of $T$. 
We define three equivalence relations $\equiv^{0}$, $\equiv^1$, and $\equiv^2$ on $\SYT(\lambda)$, each a refinement of the next. One particularly nice equivalence class of  $\equiv^{0}$ and $\equiv^{1}$ is the class of a single element, the \emph{superstandard tableau} $U_\lambda$. Here, $\lambda$ is a partition of $n$ and $U_\lambda\in\SYT(n)$ is the standard Young tableau attained by filling each row of $\lambda$, in order, with as small of values as possible. The main result can then be stated as follows.

\begin{thm}\label{symmetric to Schur}
For $k=0,1,2,$ let $\mc\subset \SYT(n)$ be the disjoint union of equivalence classes of $\equiv^{k}$ such that ${f=\sum_{T\in \mc} F_{\ID(T)}}$ is a symmetric function. Then  
\[ f=\sum_{\lambda\vdash n} c_\lambda s_\lambda, \]
\noindent
where $c_\lambda$ is the multiplicity of $U_\lambda$ in $\mc$.
\end{thm}

\noindent Corollary~\ref{positive perms} then generalizes this theorem to equivalence classes of permutations, where each relation becomes a refinement of dual Knuth equivalence. All of these relations on permutations has the added benefit of commuting with Knuth equivalences, (or, equivalently, jeu de taquin), as stated in Proposition~\ref{Kj commutes}. In Conjecture~\ref{f2 basis},  we further propose that the family of fundamental quasisymmetric generating functions over $\equiv^2$, $\{f^{(2)}\}$, form a basis for the quasisymmetric functions.

We then turn our attention to two applications. First is the set of quasisymmetric functions, which --- as the name suggests --- are a quasisymmetric analogue of the Schur functions. They were introduced in \citep{haglund2011quasisymmetric} and have since been studied for there relation to Demazure atoms, the Littlewood-Richardson rule, the RSK correspondence, and 0-Hecke algebras, amongst others (see \citep{Luoto2013introduction} for an overview of the topic). 
Quasisymmetric Schur functions can be realized as a sum over standard reverse composition tableaux of a fixed composition shape, $\SRCT(\alpha)$.
 In Proposition~\ref{D transitive}, we describe a transitive action on $\SRCT(\alpha)$. Using a result in \citep{mason2008decomposition}, $\SRCT(\alpha)$ can be mapped to a subset of the standard reverse tableaux of a fixed partition shape $\SRT(\lambda)$. We show in Corollary~\ref{bar d transitive} that Mason's bijection sends our transitive action on $\SRCT(\alpha)$ to a subset of dual Knuth equivalences, which we term \emph{quasi-dual equivalences}. In Lemma~\ref{bar d coarse}, we show that the quasi-dual equivalence is a coarsening of $\equiv^k$, for each $k=0, 1, 2$. Hence the quasisymmetric Schur functions are positive integer sums of functions in $\{f^{(k)}\}$, as stated in Theorem~\ref{bar d coarse}.

Finally, we consider shifted dual equivalence. Our presentation is most closely related to \citep{Haiman}, though it was was originally studied simultaneously in \citep{sagan1987shifted} and \citep{worley1984theory}. Shifted dual equivalence is related to enumerative properties of reduced words in Lie type B, Stanley symmetric functions, the $P$ and $Q$-Schur functions, and shifted dual equivalence graphs (also see \citep{Billey2014Coxeter}, \citep{stanley1984number}, and \citep[Ch. 7]{Stanley2}). 
Similar to with the quasisymmetric Schur functions, Proposition~\ref{restricted are shifted} states that the generators for shifted dual equivalence strictly contain the generators for $\equiv^2$ after applying a simple involution. Hence, the fundamental quasisymmetric generating functions over shifted dual equivalence classes may be added to our poset, as stated in Theorem~\ref{fh positive}. In Proposition~\ref{ssyt are equiv2}, we further show that the set of row reading words of shifted standard Young tableaux with a fixed shape comprise an equivalence class of $\equiv^2$.

The paper is organized as follows. After preliminary lemmas and definitions, Section~\ref{relations} defines the equivalence relations $\equiv^k$ and proves Theorem~\ref{symmetric to Schur}. Section~\ref{equiv on perms} generalizes results to permutations, with the main result generalized in Corollary~\ref{equiv on perms}.
 Section~\ref{quasis} is dedicated to the previously mentioned results related to quasisymmetric Schur functions. Finally, Section~\ref{shifted}, which is independent of Section~\ref{quasis}, is dedicated to proving the results related to shifted dual equivalence.

\section{The equivalence relations}\label{relations}

\subsection{Preliminaries}
A composition $\alpha=(\alpha_1, \alpha_2, \ldots, \alpha_k)$ is a finite sequence of positive integers. We write $\alpha\vDash n$ and say that $\alpha$ is a \emph{composition of $n$} if $\sum_i \alpha_i=n$. A partition $\lambda=(\lambda_1, \lambda_2, \ldots, \lambda_k)$ is a weakly decreasing composition. We write $\lambda \vdash n$ and say $\lambda$ is a partition of $n$ if $\sum_i \lambda_i =n$. In this paper, $\lambda$ will always be a partition of $n$,  $\alpha$ will always be a composition of $n$, and $\lambda(\alpha)$ is the composition achieved by sorting the parts of $\alpha$ into weakly decreasing order. 

The \emph{diagram} of a composition or a partition will always be given in French notation. That is, the diagram is given by a set of left justified \emph{cells}, drawn as boxes, in the Cartesian plane, where the $i^{th}$ row from bottom to top has $\alpha_i$  or $\lambda_i$ cells, respectively. The bottom left cell can be assumed to be the origin. The underlying composition or  partition is then referred to as the \emph{shape} of the diagram. Given a partition $\lambda$, the \emph{conjugate} of $\lambda$, denoted $\lambda'$, is the partition whose $i^{th}$ column has $\lambda_i$ many cells.

 A \emph{filling} $T$ of a diagram is a function that assigns a positive integer to each box of the diagram.  A filling is \emph{standard} if it uses every number in $[n]=\{1, \ldots, n\}$ exactly once.
The set of \emph{standard Young tableaux} of shape $\lambda$, denoted $\SYT(\lambda)$, is the set of standard fillings of $\lambda$ that are increasing up columns and across rows from left to right. The union of $\SYT(\lambda)$ over all $\lambda\vdash n$ is denoted $\SYT(n)$. The \emph{superstandard tableau} of shape $\lambda$, denoted $U_\lambda$, is the standard Young tableau formed by placing the numbers 1 through $n$, in order, in the lowest row possible.
The \emph{row reading word} of a filling, $\rw(T)$, is obtained by reading the values of each row from left to right, starting with the top row and continuing down. The row reading word will always be used for standard Young tableaux, as well as for ordering their cells, but other reading words will be introduced for other types of fillings as needed.

In this paper, permutations are always given in one-line notation. Given a permutation $\pi \in S_n$, the \emph{inverse descent set} of $\pi$, $\ID(\pi)$, is the set of $i\in [n-1]$ such that $i$ occurs after $i+1$ in $\pi$. 
 Alternately, we can encode $\ID(\pi)$ with the composition $\beta(\pi)=(\beta_1(\pi), \dots,\beta_k(\pi)) \vDash n$, letting $\beta_i$ be the difference between the $i$ and $i-1^{th}$ numbers in $\ID(\pi)\cup\{n\}$, where the zero$^{th}$ number is always treated as $0$. Given a standard filling $T$, define $\ID(T)$ and $\beta(T)$ via the reading word of $T$. Define the \emph{$i^{th}$ run of $T$}, $T|_{(a,b]}$, to be the restriction of $T$ to the set of cells with values in the integer interval $(a,b]$, where $a$ and $b$ are the $i-1$ and $i^{th}$ numbers in $\ID(T)\cup \{n\}$, respectively. Notice that $b-a=\beta_i$. Similarly, the \emph{first $j$ runs of $T$} is the standard Young tableau achieved by taking the union of the first $j$ runs of $T$. 
Notice that for any $T\in\SYT(\lambda)$, the cells of the runs of $T$ fully determine $T$. For this reason, it can be helpful to consider the \emph{unstandard Young tableau}, UYT, achieved by placing $i$'s in each cell of the $i^{th}$ run. See Figure~\ref{UYT} for examples.

 \ytableausetup{smalltableaux,aligntableaux=center,mathmode}
\begin{figure}[h]
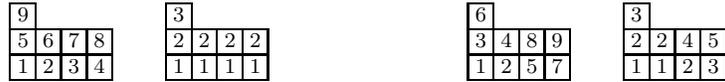

\[\begin{array}{ccccccccc}
\ytableaushort{9,5678,1234}&& \ytableaushort{3,2222,1111} &&\quad\quad\quad&&
\ytableaushort{6,3489,1257}&& \ytableaushort{3,2245,1123}
\end{array}\]
\caption{Two tableaux in $\SYT(\lambda)$ with their associated fillings in UYT$(\lambda)$ to their right. Here, $\lambda=(4, 4, 1)$. At left, $U_\lambda$ with $\rw(U_\lambda)=956781234$, $\ID(U_\lambda)=\{4, 8\}$, and $\beta(T)=(4,4,1)$. At right, a tableau $T$ with $\rw(T)=634891257$, $\ID(T)= \{2, 5, 7\}$ and $\beta(T)=(2,3,2,2)$.} \label{UYT}
\end{figure}

The RSK correspondence provides a bijection that sends each $\pi\in S_n$ to an ordered pair of standard Young tableaux $(P(\pi), Q(\pi))$, where $P(\pi)$ and $Q(\pi)$ are both in $\SYT(\lambda)$ for some $\lambda\vdash n$. Here, $P(\pi)$ is termed the \emph{insertion} tableau and $Q(\pi)$ is termed the \emph{recording tableau}.  Two permutations with the same $P$ tableau are \emph{Knuth equivalent}, while two permutations with the same $Q$ tableau are \emph{dual Knuth equivalent} or just \emph{dual equivalent}. The map from $\pi$ to $(P(\pi), Q(\pi))$ can be achieved via the `row bumping algorithm' or `jeu de taquin' and is more generally defined on words, but we will restrict our attention to permutations. We will assume familiarity with the RSK correspondence, though a treatment can be found in \citep{sagan}.

Let $\phi$ be a function on standard Young tableaux that restricts to an involution on each $\SYT(\lambda)$. Define the action of $\phi$ on $S_n$  \emph{via the insertion tableau} to be the function that sends $\pi\in S_n$ to the unique $\pi'$ such that $P(\pi')=\phi\circ P(\pi)$ and $Q(\pi')=Q(\pi)$.

In the context of permutations, the role of $U_\lambda$ will be replaced by the set of \emph{standardized Yamanouchi words} of shape $\lambda$, defined as 
\begin{equation}
\SYam(\lambda)=\{\pi \in S_n \colon P(\pi)=U_\lambda\}
\end{equation}
\noindent
As an aside, we may also algorithmically generate $\SYam(\lambda)$. First generate the set of words with $\lambda_i$ many $i$'s such that when reading in reverse order there are always weakly more $i$'s than $i+1$'s. These are the  Yamanouchi words of weight $\lambda$. To achieve $\SYam(\lambda)$, turn each word into a permutation with the same relative order, where an $i$ that occurs earlier in reading order is considered smaller than an $i$ with a larger index.

We may generate \emph{dual equivalence classes} of permutations as follows.
Given a permutation in $S_n$ expressed in one-line notation, define an \emph{elementary dual equivalence} as an involution $d_i$ that interchanges the values $i-1, i,$ and $i+1$ as
\begin{equation}\label{dual equiv}
\begin{split}
d_i(\ldots i \ldots i-1 \ldots i+1 \ldots) = (\ldots i+1 \ldots i-1 \ldots i \ldots),\\
d_i(\ldots i \ldots i+1 \ldots i-1 \ldots) = (\ldots i-1 \ldots i+1 \ldots i \ldots),
\end{split}
\end{equation}
and as the identity when $i$ is between $i-1$ and $i+1$.
Two words are then dual equivalent if one may be transformed into the other by successive elementary dual equivalences.
As an example, 21345 is dual equivalent to 51234 because
$d_4(d_3(d_2( 21345 ))) = d_4(d_3( 31245 )) = d_4( 41235 ) = 51234$. 

For a permutation $\pi \in S_{n}$, let $\pi|_I$ be the subword consisting of values in the interval $I$.  Let $\fl(\pi|_I) \in
S_{|I|}$ be the permutation with the same relative order as $\pi|_{I}$. Here $\mathrm{fl}$ is the \textit{flattening operator}.  For example $\fl(31245|_{[2,4]})=\fl(324)=213$.
By considering the action of $d_i$ on $\fl(\pi|_{[i-1,i+1]})$, we may express all nontrivial actions as
\begin{equation}\label{flat dual}
x1y \quad  \quad 
x3y,
\end{equation}
\noindent
where $d_i$ swaps the values $x$ and $y$ in $\fl(\pi|_{[i-1,i+1]})$ while fixing all values of $\pi$ not in $[i-1, i+1]$.

Elementary dual equivalences act on standard Young tableaux via their row reading words, as shown in Figure~\ref{dual tableau}. The next theorem expresses the transitivity of this action.

\begin{figure}[h]
\[
\xymatrix{
\ytableaushort{2,1345} \ar@{<->}[r]^{d_2} & \ytableaushort{3,1245} \ar@{<->}[r]^{d_3}&\ytableaushort{4,1235}  \ar@{<->}[r]^{d_4}&\ytableaushort{5,1234}
}
\]
\caption{The action of $d_i$ on $\SYT((4,1))$.}\label{dual tableau}
\end{figure}
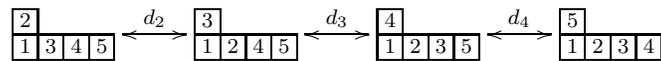

\begin{thm}[{\cite[Prop.~2.4]{Haiman}}]\label{Haiman}
Two standard Young tableaux on partition shapes are dual equivalent if and only if they have the same shape.
\end{thm}

\noindent
 In fact, $d_i$ is defined so that 
\begin{equation}\label{d P commute}
P\circ d_i(\pi)=d_i\circ P(\pi) \quad \textup{and} \quad Q\circ d_i(\pi)=Q(\pi).
\end{equation}

\noindent
Furthermore
\begin{equation}\label{ID and P commute}
\ID(P(\pi))=\ID(\pi).
\end{equation}

The \emph{Knuth equivalences} act similarly on permutations, with the roles of $P$ and $Q$ reversed. Define the generator 
\begin{equation}
K_i(\pi)=(d_i(\pi^{-1}))^{-1}.
\end{equation}
More explicitly, to perform $K_i$, consider $\pi_{[i-1, i+1]}$ written using the values $x<y<z$. Then $K_i$ permutes the values in $\pi_{[i-1, i+1]}$ by taking $yxz$ to $yzx$; $xzy$ to $zxy$; and fixing $xyz$ as well as $zyx$. For example, $K_4(42153)=42513$. These involutions interact with the RSK correspondence by
\begin{equation}\label{K Q commute}
P\circ K_i(\pi)=P(\pi) \quad \textup{and} \quad Q\circ K_i(\pi)=d_i\circ Q(\pi).
\end{equation}

\noindent
The previous facts about dual equivalence were shown in \citep{Haiman}, while a more general treatment may be found in \citep{sagan}.

We will be using tableaux as an enumerative tool for symmetric function calculations. Traditionally, the symmetric functions are defined as the fixed set of formal power series in variables $x_1, x_2, \ldots$ under permutations of their indices. For our purposes, it is enough to define the symmetric functions of degree $n$ as the real vector space generated by the basis of Schur functions $s_\lambda$, where $\lambda \vdash n$. We will additionally care about the spanning set of \emph{composition Schur functions}, $s_{\alpha}$, where $\alpha \vDash n$. Each $s_\alpha$ is equal to $0$, $s_\lambda$, or  $-s_\lambda$ for a unique $\lambda\vdash n$ attained from $\alpha$. To be more specific, if $\beta$ is the result of taking $\alpha$ and replacing some $(\alpha_i, \alpha_{i+1})$ with $(\alpha_{i+1}-1, \alpha_i +1)$, then $s_\alpha=-s_\beta$. In particular, if there exists an $i$ such that $\alpha_i+1=\alpha_{i+1}$, then $s_\alpha=0$. If  $s_{\alpha}\neq 0$, a series of such swaps of $\alpha_i$ and $\alpha_{i+1}$ may always be used to transform $\alpha$ into a unique partition $\lambda$. 

The relationship between composition Schur functions and (partition) Schur functions is sometimes called the \emph{slinky correspondence} 
  for the following reason. If $\alpha$ is represented by a diagram and $s_\alpha \neq 0$, then we may relate $s_\alpha$ to a unique $s_\lambda$, where $\lambda\vdash n$ by letting gravity pull each row down to create a partition shape, as in Figure~\ref{slinky correspondence}. If this process does not yield a partition, the result is 0. Otherwise, the partition shape is $\lambda$, and the sign in front of the Schur function $s_\lambda$ is $(-1)^{d}$, where $d$ is the sum of the number or rows each part of the composition was shifted down to make the partition. The term \emph{slinky} is meant to evoke how each row \emph{slinks} downward.

\begin{figure}[h]
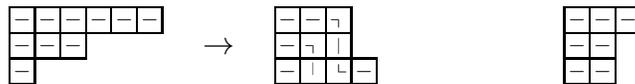

 \ytableausetup{smalltableaux,aligntableaux=center,mathmode}
\[ \ytableaushort{------,---,- }
\quad\to\quad
\ytableaushort{--{\raisebox{-.054cm}{$_\urcorner$}\hspace{.076cm}},-{\raisebox{-.054cm}{$_\urcorner$}\hspace{.076cm}}{\raisebox{.018cm}{\scalebox{.51}{$|$}}},-{\raisebox{.06cm}{\scalebox{.4}{{\bf $|$}}}}{\raisebox{.13cm}{$_\llcorner$}\hspace{-.12cm}}-}
\quad\quad\quad\quad\quad\quad
\ytableaushort{---,--,--}\]
\caption{The  slinky correspondence showing $s_{(1,3,6)}= (-1)^3 s_{(4,3,3)}=-s_{(4,3,3)}$ and $s_{(2,2,3)}=0$.}\label{slinky correspondence}
\end{figure}

The following relationship between fundamental quasisymmetric functions and composition Schur functions will prove crucial for our purposes.

\begin{lem}[\citep{egge2010quasisymmetric}]\label{comp Schur fund}
If $f$ is a symmetric function such that 
\begin{align}
f=\sum_{\alpha \vDash n} c_\alpha F_\alpha, \quad \textup{then} \quad f=\sum_{\alpha \vDash n} c_\alpha s_\alpha,
\end{align}
\noindent where $c_\alpha$ are constants.
\end{lem}

\subsection{The equivalence relations on tableaux}
 
 \begin{dfn} \label{slink} \textup{
Consider $T\in\SYT(\lambda)$. Define $\slink(T)$  and $\slink_*(T)$ as follows. 
\begin{itemize}
\item Let $j$ be the minimal number such that the first $j$ runs of $T$ do not form a superstandard tableau. If such a $j$ does not exist $(T=U_\lambda)$, then $\slink$ and $\slink_*$ act as the identity.
\item Construct $\slink(T)\in \SYT(\lambda)$ by permuting the cells of the $j-1$ and $j^{th}$ runs that occur below the $j^{th}$ row, giving the first $\beta_{j}(T)-1$ of these cells to the $j-1^{th}$ run. In particular,
\begin{equation}\label{beta slink}
\beta_{j-1,j}(\slink(T))=(\beta_j(T)-1,\beta_{j-1}(T)+1).
\end{equation}
\item Let $\mu$ be the shape of the first $j$ runs of $T$, and let $i$ be the minimal number such that
 \begin{equation}\label{mu condition}
\mu_{i+1}(T) \leq \beta_{j}(T)+i-j.
 \end{equation}
\item Construct $\slink_*(T)\in \SYT(\lambda)$ by permuting the cells of the $i^{th}$ and $j^{th}$ runs that occur below the $j^{th}$ row, giving the first $\beta_{j}(T)+i-j$ of these cells the $i^{th}$ run. In particular,
\begin{equation}\label{ij star}
\beta_{i,j}(\slink_*(T))=(\beta_j(T)+i-j,\beta_i(T)+j-i).
\end{equation}
\end{itemize}
}
\end{dfn}
%
%
%
%
\noindent Notice that the definition above refers to permuting cells in runs, dictating what the values in those cells will have to be. See Figure~\ref{slink example} for examples.
A few ramifications of this definition should be pointed out. 
%
\begin{lem}\label{slink consequence}
For any partition $\lambda$ and $T\in\SYT(\lambda)$, the following hold.
\begin{enumerate}
\item $\slink(T)$ and $\slink_*(T)$ are well defined and in $\SYT(\lambda)$.
\item\label{slink sign} If $T\neq U_\lambda$, then $s_{\beta(\slink(T))}= -s_{\beta(T)}$.
\item If $i=j-1$ in Definition~\ref{slink}, then $\slink(T)=\slink_*(T)$
\end{enumerate}
\end{lem}
\begin{proof}
The proofs are direct consequences of the definition and are recommended for the reader. 
\end{proof}

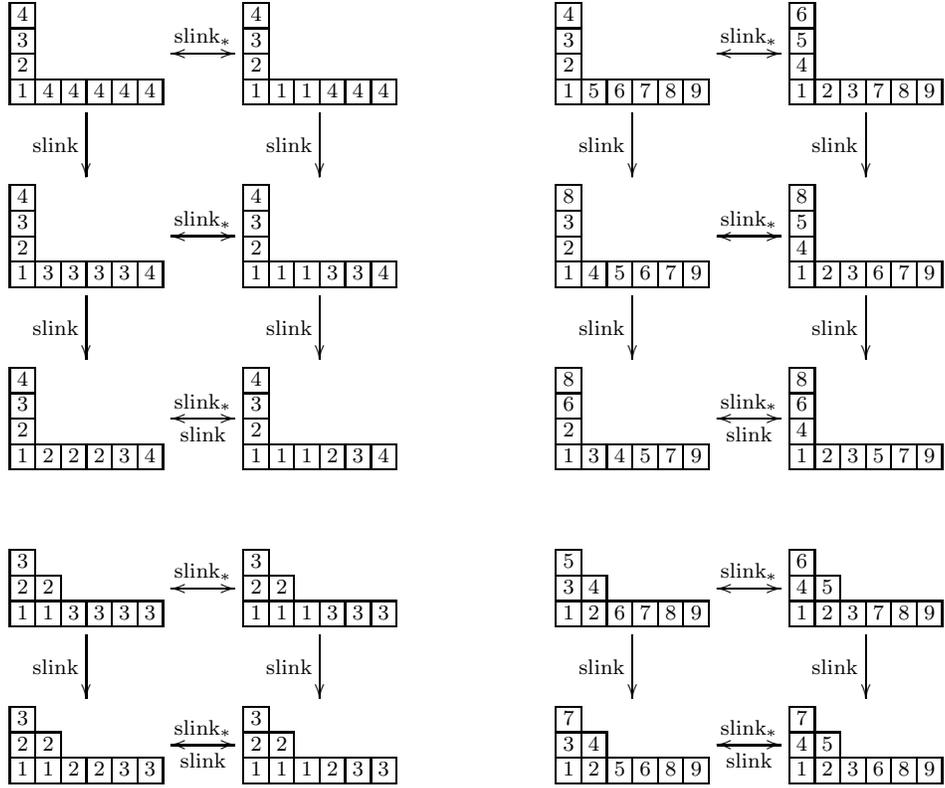
\begin{figure}[h]
  \ytableausetup{smalltableaux,aligntableaux=center}
\[
  \xymatrix{ 
    \ytableaushort{4, 3, 2, 144444} \ar@{<->}[r]^{\slink_*} \ar@{->}[d]_{\slink} & \ytableaushort{4, 3,2, 111444}  \ar@{->}[d]_{\slink} &&
     \ytableaushort{4, 3, 2, 156789} \ar@{<->}[r]^{\slink_*} \ar@{->}[d]_{\slink}& \ytableaushort{6, 5, 4, 123789}  \ar@{->}[d]_{\slink}
     \\
      \ytableaushort{4, 3, 2, 133334}  \ar@{<->}[r]^{\slink_*} \ar@{->}[d]_{\slink}  & \ytableaushort{4, 3, 2, 111334}  \ar@{->}[d]_{\slink}&&
      \ytableaushort{8, 3, 2, 145679}  \ar@{<->}[r]^{\slink_*} \ar@{->}[d]_{\slink}  & \ytableaushort{8, 5, 4, 123679}  \ar@{->}[d]_{\slink}
      \\
  \ytableaushort{4, 3, 2, 122234}  \ar@{<->}[r]^{\slink_*}_{\slink} & \ytableaushort{4, 3, 2, 111234} &&
  \ytableaushort{8, 6, 2, 134579}  \ar@{<->}[r]^{\slink_*}_{\slink} & \ytableaushort{8, 6, 4, 123579}  
  \\
    \ytableaushort{3, 22, 113333} \ar@{<->}[r]^{\slink_*} \ar@{->}[d]_{\slink} &\ytableaushort{3, 22, 111333}  \ar@{->}[d]_{\slink}&&
      \ytableaushort{5, 34, 126789} \ar@{<->}[r]^{\slink_*} \ar@{->}[d]_{\slink} &\ytableaushort{6, 45, 123789}  \ar@{->}[d]_{\slink}
     \\
   \ytableaushort{3, 22, 112233}  \ar@{<->}[r]^{\slink_*}_{\slink} & \ytableaushort{3, 22, 111233}& &
   \ytableaushort{7, 34, 125689}  \ar@{<->}[r]^{\slink_*}_{\slink} & \ytableaushort{7, 45, 123689} 
   \\
 }
  \]
\caption{Two equivalence classes of $\equiv^{1}$ using $\SYT(\lambda)$ on the right and their representation in UYT$(\lambda)$ on the left. The $\equiv^0$ classes are the vertex sets of edges labeled $\slink_*$.}\label{slink example}
\end{figure}

We may now define two equivalence classes that will motivate much of the rest of this paper. 

\begin{dfn}\label{def of equiv}
\textup{
Define $\equiv^{0}$ and $\equiv^1$ as the equivalence relations on standard Young tableaux generated by $\slink_*$ and $\slink$, respectively.
}
\end{dfn}

\noindent
Examples of equivalence classes can be found in Figure~\ref{slink example}. The relationships between $\slink_*$ and $\equiv^1$ are further illuminated by the following lemmas.

\begin{lem}\label{ij}
For any partition $\lambda$ and $T\in \SYT(\lambda)$ such that $T\neq U_\lambda$, let $i$ and $j$ be as in Definition~\ref{slink}. 
\begin{enumerate} 
\item\label{ij1} If $i=j-1$, then $\slink$ fixes $i$ and $j$,
\item\label{ij2} If $i<j-1$, then $\slink$ fixes $i$ and lowers $j$ by 1,
\item\label{ij3} $\slink_*$ fixes $i$ and $j$.
\end{enumerate}
\end{lem}
\begin{proof} Let $T, i$ and $j$ be as in the statement of the Lemma. Let $\mu$ be as in Definition~\ref{slink}, where we use $\mu(T), \mu(\slink(T))$, and $\mu(\slink_*(T))$ to differentiate the possibly different shapes associated to each tableau. Consider the three parts of the lemma in order.

\noindent \emph{Part~\ref{ij1}}: 
Assume $i=j-1$. By Part~3 of Lemma~\ref{slink consequence}, $\slink(T)=\slink_*(T)$. The result is thus a special case of Part~\ref{ij3}, which is proven below.

\noindent \emph{Part~\ref{ij2}}:
Assume $i<j-1$. By (\ref{mu condition}) and (\ref{beta slink}),
\begin{equation}
\mu_i(T)\leq \beta_j(T)+i-j< \beta_j(T)-1=\beta_{j-1}(\slink(T)).
\end{equation}
Thus, the $j-1^{th}$ run of $\slink(T)$ must take all of the cells of the $j^{th}$ run of $T$ that occur weakly between the $i$ and $j-1^{th}$ rows, as well as at least one cell below the $i^{th}$ row. Hence, the first $j-1$ runs of $\slink(T)$ do not form a superstandard tableau, and the $j$ value must be lowered by 1. 
Further, $\mu_{i,i+1}(T)=\mu_{i,i+1}(\slink(T))$. 

To show that $\slink$ fixes $i$, it suffices to use (\ref{beta slink}) to inspect (\ref{ij star}) for both $i$ and $i-1$,
\begin{align}
\begin{split}
\mu_{i+1}(\slink(T))= & 
\mu_{i+1}(T) 
\leq  \beta_{j}(T)+i-j \\
=&  \beta_{j-1}(\slink(T))+i-(j-1),
\end{split}
\end{align}
\begin{align}
\begin{split}
\mu_{i}(\slink(T))= & 
\mu_{i}(T) 
>  \beta_{j}(T)+(i-1)-j \\
=&  \beta_{j-1}(\slink(T))+(i-1)-(j-1).
\end{split}
\end{align}
Thus, $i$ is still the minimal integer satisfying (\ref{mu condition}).

\noindent \emph{Part~\ref{ij3}}: If $\slink_*(T)=T$, the result is trivial. Assume otherwise. We start by showing that $\slink_*$ does not change $j$. 
By (\ref{ij star}) and (\ref{mu condition}),
\begin{equation}
\beta_j(\slink_*(T))=\beta_i(T)+j-i >\beta_i(T) \geq \beta_{j-1}(T)\geq \mu_j(T),
\end{equation}
so the $j^{th}$ run of $\slink_*(T)$ must have cells below the $j^{th}$ row. That is, the first $j$ runs of $\slink_*(T)$ do not form a superstandard tableau. To prove that the first $j-1$ runs of $\slink_*(T)$ form a superstandard tableau, consider two cases. 

First,
suppose that $\slink_*$ increases the size of the $i^{th}$ run of $T$. In this case, we need to show that there are no new cells of the $i^{th}$ run strictly below the $i^{th}$ row, guaranteeing that the first $j-1$ runs comprise a superstandard tableau. Suppose the contrary. By the definition of $\slink_*$, the $i^{th}$ run of $\slink_*(T)$ would have to contain $\mu_i(T)$ cells in the $i^{th}$ row as well as at least one more cell in a lower row. Using this fact and (\ref{ij star}), 
\begin{equation}
\mu_i \leq \beta_i(\slink_*(T))-1=\beta_{j}(T)+(i-1)-j.
\end{equation}
In particular, $i$ was not chosen minimally, providing the desired contradiction. 

Second, suppose $\slink_*$ decreases the size of the $i^{th}$ run. The first $j-1$ runs of $\slink_*(T)$ are then the result of removing cells from the $i^{th}$ run of a superstandard tableaux. Applying Part~1 of Lemma~\ref{slink consequence}, this must still be a standard Young tableaux. Hence, the first $j-1$ runs of $\slink_*(T)$ form a standard Young tableaux while the first $j$ runs do not. That is, $\slink_*$ preserves $j$.


We still need to show that $\slink_*$ does not change $i$. We proceed by checking that $i$ is the minimal number satisfying (\ref{mu condition}) for $\slink_*(T)$. Since $\slink_*$ preserves $j$ and permutes the cells of the first $j$ runs, it must also preserve $\mu(T)$. Further, $\mu_{i+1}(T)$ must be weakly less than $\beta_i$, since the first $j-1$ runs of $T$ comprise a superstandard tableau. Hence
\begin{align}
\mu_{i+1}(\slink_*(T))=\mu_{i+1}(T) \leq \beta_i(T)=\beta_j(\slink_*(T))+i-j. 
\end{align}

\noindent
To show that $i$ is still minimal, we proceed by contradiction. Suppose that $i-1$ also satisfies (\ref{mu condition}) for $\slink_*(T)$. Then
\begin{eqnarray}
\begin{split}
\mu_{i}(T) \geq & \beta_i(T)  
 =\beta_j(\slink_*(T))+i-j   \\
 >& \beta_j(\slink_*(T))+(i-1)-j   
\geq \mu_i(\slink_*(T)) \\
\geq& \beta_i(\slink_*(T))
 = \beta_{j}(T)+i-j. 
\end{split}
\end{eqnarray}
Hence, $i$ was not the minimal choice satisfying (\ref{mu condition}) for $T$. 
This contradiction concludes the proof of Part~\ref{ij3}.
\end{proof}

\begin{prop}\label{star involution}
For any $T\in \SYT(n)$, $\slink_* \circ \slink_*(T)=T$.
\end{prop}
\begin{proof}
If $\slink_*$ acts as the identity, then the result it trivial, so assume otherwise.
Part~\ref{ij3} of Lemma~\ref{ij} guarantees that $\slink_*$ fixes $i$ and $j$  in Definition~\ref{slink}. The permutation of runs described in (\ref{ij star}) is thus an involution, completing the proof.
\end{proof}

\begin{lem}\label{slink commute} 
For any $T\in\SYT(n)$, 
$\slink\circ \slink_{*}(T) = \slink_{*}\circ\slink(T)$.
\end{lem}
\begin{proof}
If $\slink(T)=T$, then $\slink_*(T)=T$, and the proof is trivial. We may then assume that $\slink(T)\neq T$, and let $i$ and $j$ be as in the definition of $\slink_*(T)$. If $i=j-1$, then $\slink_*(T)=\slink(T)$ by Part~3 of Lemma~\ref{slink consequence}, so $\slink_*\circ \slink(T)=\slink_*\circ \slink_*(T)=T$,  by Proposition~\ref{star involution}. By Part~\ref{ij3} of Lemma~\ref{ij}, $\slink_*$ preserves the $i$ and $j$ values from Definition~\ref{slink}. Hence, $\slink \circ \slink_*(T)=\slink_*\circ \slink_*(T)=T$, completing the $i=j-1$ case.

 Now assume that $i<j-1$. Applying Lemma~\ref{ij} to $\slink \circ \slink_*(T)$, $\slink_*$ effects the $i$ and $j^{th}$ runs, and then $\slink$ effects the $j$ and $j-1^{th}$ runs. Conversely, in $\slink_* \circ \slink(T)$, $\slink$ effects the $j-1$ and $j^{th}$ runs, and then $\slink_*$ effects the $i$ and $j-1^{th}$ runs. In particular, both $\slink\circ \slink_{*}$ and $\slink_{*}\circ\slink$ only effects the cells of the $i, j-1,$ and $j^{th}$ runs. By (\ref{beta slink}) and (\ref{ij star}), the action  on $\beta(T)$ gives

\begin{align}
\begin{split}
 \beta_{i,j-1,j}(\slink \circ \slink_*(T)) &=
  \beta_{i,j-1,j}(\slink_*\circ\slink(T))\\ 
&=(\beta_j(T)+i-j,\beta_i(T)+j-i-1,\beta_{j-1}(T)+1).
\end{split}
\end{align}

Again,  $\slink\circ \slink_{*}(T)$ and  $\slink_{*}\circ\slink(T)$ only differ in the locations of the $i$, $j-1$, and $j^{th}$ runs, but the number of cells in these runs are equal. It thus suffices to show that the $i^{th}$ and $j^{th}$ runs are in the same cells, forcing the $j-1^{th}$ run to also agree.
By Lemma~\ref{ij}, the first $i$ runs of both form superstandard tableaux. Hence, the locations of the $i^{th}$ run is fully determined by its inverse descent set, and so must be the same.
Now let $a$ be the number of cells in the $j^{th}$ run of $T$ that are in the $j^{th}$ row of $T$. Applying Definition~\ref{slink}, in both $\slink\circ \slink_{*}(T)$ and $\slink_{*}\circ\slink(T)$, the $j^{th}$ run will consist of the $a$ cells in the $j^{th}$ row and the last $\beta_{j-1}(T)+1-a$ cells of the $j^{th}$ run of $T$.  Hence, both tableaux have the $j^{th}$ runs in the same cells, completing the proof.
\end{proof}

\begin{prop}\label{slink star}
Let $T\in \SYT(n)$. Then 
$\slink^{j-i}(T)=\slink^{j-i-1}\circ \slink_*(T)$,  where $i$ and $j$ are as in Definition~\ref{slink}. In particular,
$\equiv^0$ is a refinement of $\equiv^1$.
\end{prop}
\begin{proof}
We will prove the first part by induction on $j-i$. 
If $j$ does not exist, then the result is trivial. If $j-i=1$, then $\slink(T)= \slink_*(T)$, by Part~3 of Lemma~\ref{slink consequence}. Now assume the result whenever $j-i \leq k$, and assume $j-i=k+1>1$. 

We would like to apply our inductive hypothesis to $\slink(T)$. To justify this, recall that Part~\ref{ij2} of Lemma~\ref{ij} guarantees that $\slink$ fixes the $i$ value of $T$ while lowering the $j$ value by 1, hence lowering $j-i$ by 1.  
Applying our inductive hypothesis, as well as Lemma~\ref{slink commute}, gives
\begin{align}
\slink^{j-i}(T)=\slink^{j-i-1}\circ \slink(T)= \slink^{j-i-2}\circ \slink_*\circ \slink(T)=\slink^{j-i-1}\circ \slink_*(T),
\end{align}

\noindent
completing the proof of the first part of the lemma.
The second part of the proposition now follows from Definition~\ref{def of equiv}, since $\slink*(T)$ is always connected to $T$ by some sequence of applications of $\slink$.
\end{proof}
 
\begin{lem}\label{big slink}
Given any partition $\lambda$ and tableau $T\in \SYT(\lambda)$, exactly one of the following holds.

\begin{enumerate}
\item\label{slink1} $T=U_\lambda$.


\item\label{slink3} $s_{\beta(T)}=-s_{\beta(\slink_*(T))}$

\end{enumerate}
\end{lem}

\begin{proof}
First, we show that the two cases of the lemma are disjoint. If $T$ is described by Case \ref{slink1}, then $\slink_*(T)=T$ and $s_{\beta(T)}=s_\lambda\neq 0$. Hence, $s_{\beta(T)}\neq -s_{\beta(\slink_*(T))}$, and Case~\ref{slink3} does not apply.

It now suffices to assume that $T$ is not described by Case~\ref{slink1}  and show that it satisfies Case~\ref{slink3}.  By Proposition~\ref{slink star}, $T$ is connected to $\slink_*(T)$ by an odd number of applications of $\slink$. Applying Part~\ref{slink sign} of Lemma~\ref{slink consequence}, $s_{\beta(T)}=-s_{\beta(\slink_*(T))}$. 
\end{proof}

\begin{lem}\label{0 classes}
If $\mc$ is any equivalence class of $\equiv^{0}$ or $\equiv^{1}$ such that $\mc\neq \{U_\lambda\}$ for any partition $\lambda$, then $ \sum_{T \in \mc} s_{\beta(T)} = 0$.
\end{lem}
\begin{proof}
We proceed by providing a sign reversing involution. First notice that $U_\lambda$ is always alone in its $\equiv^0$ and $\equiv^1$ classes, so we may assume $U_\lambda \notin \mc$. Each $T\in \mc$ contributes one composition Schur function to the sum, which is in turn equal to 0  or $\pm s_\lambda$, for some $\lambda \vdash n$. By Proposition~\ref{slink star}, $\slink_*$ is an involution on $\mc$. By Lemma~\ref{big slink} and the assumption that $\mc\neq \{U_\lambda\}$, $s_{\beta(T)}+s_{\beta(\slink_*(T))}=0$. Thus, the sum of $s_{\beta(T)}$ over $T\in \mc$ must equal zero.
\end{proof}

Next, we add another equivalence relation, which we term \emph{restricted dual equivalence}.

\begin{dfn}\label{restricted d}
\textup{For any $T\in \SYT(n)$ and any $i\in [2,n-2]$, let}
\begin{equation*} 
d^{R}_i(T) =
 \begin{cases}
 T & i+1\in \ID(T)\cap \ID(d_i(T)),\\
 d_i(T) & \textup{otherwise.}
\end{cases}
\end{equation*}
\textup{Further, let $\equiv^2$ be the equivalence relation on standard Young tableaux generated by the action of all $d_i^R$.}
\end{dfn}


\begin{thm}\label{restricted positive}
The equivalence relation $\equiv^0$ is a refinement of $\equiv^1$, which in turn refines $\equiv^2$.
\end{thm}
\begin{proof}
The first part of the statement follows directly from Proposition~\ref{slink star}, so we need only show that $\equiv^1$ refines $\equiv^2$. It suffices to show that the action of $\slink$ on $T\in \SYT(n)$ can always be achieved by a series of $d_i^R$. 

First consider the case where $T\in\SYT(\lambda)$ and $\lambda$ has at most two rows. It follows directly from the Definition~\ref{slink} that $\slink$ fixes the cell containing $n$, which is the only property we will need. In the two row case, $\{i, i+1\}$ cannot be a subset of $\ID(T)$, because $i+2$ would have to be two rows above $i$. Direct inspection of  the definition of $d_i$ or (\ref{flat dual}) shows that $i+1\in \ID(T)\cap\ID(d_i(T))$ implies $\{i, i+1\}\subset \ID(T)$ or $\{i, i+1\}\subset \ID(d_i(T))$. Thus, $d_i^R=d_i$ whenever $i\neq n-1$. By Lemma~\ref{Haiman}, $d_i$ acts transitively on standard Young tableaux of a fixed shape, so $d_i^R$ acts transitively on the set of tableaux in $\SYT(\lambda)$ with $n$ in a fixed cell. Thus, some series of $d_i^R$ yields the action of $\slink$ on $T$.

Next, suppose that $T$ has more than two rows. If $\slink(T)=T$, the result is trivial. Otherwise, $\slink$ acts by changing the $j-1$ and $j^{th}$ runs. Let $[a,b]$ be the values in these runs, and consider the action of $d_i^R$ on $\fl(\rw(T)|_{[a,b]})$. In particular, this is the row reading word of some two row standard Young tableau. Applying the two row case, we are free to apply $\slink$ to this two row tableau via some series of $d_i^R$. Adding $a-1$ to each $i$ in this sequence provides the sequence of $d_i^R$ to apply to $T$ in order to achieve $\slink(T)$.
\end{proof}

\begin{lem}\label{beta to lambda}
If $\mc$ is any union of equivalence classes of $\equiv^k$ with $k=0,1,2$, then 
\[\sum_{T\in \mc} s_{\beta(T)} = \sum_{U_\lambda \in \mc} s_\lambda.\]
\end{lem}

\begin{proof}
%
Applying Theorem~\ref{restricted positive}, it suffices to only consider $\equiv^0$. By Lemma~\ref{0 classes}, we may further omit any equivalence classes not of the form $\{U_\lambda\}$.
Thus,
\begin{align}
\sum_{T\in \mc} s_{\beta(T)} =  \sum_{U_\lambda \in \mc} s_{\beta(U_\lambda)} =\sum_{U_\lambda \in \mc} s_\lambda.\end{align}
\end{proof}

\begin{proof}[{\bf Proof of Theorem~\ref{symmetric to Schur}}]
For $k=0,1,2$, let $\mc$ be the disjoint union of equivalence classes of $\equiv^{k}$ such that ${f=\sum_{T\in \mc} F_{ID(T)}}$ is a symmetric function.
By Lemmas~\ref{comp Schur fund} and \ref{beta to lambda},   
\begin{equation}
 f=\sum_{T\in \mc} s_{\beta(T)}=\sum_{\lambda\vdash n} c_\lambda s_\lambda. 
 \end{equation}
where $c_\lambda$ is the multiplicity of $U_\lambda$ in $\mc$.
\end{proof}

\subsection{The equivalence relations on permutations}\label{equiv on perms}

We extend these results to permutations via the RSK correspondence. Recall that a function defined on standard Young tableaux acts on permutations  via insertion tableaux by using the RSK correspondence, acting on the insertion tableaux and fixing the recording tableaux.

\begin{dfn}
\textup{
 Let $\slink, \slink_*$, and $d_i^R$ act on permutations via insertion tableaux. For $k=0,1,2$, let $\pi \equiv^k \pi'$ if $P(\pi)\equiv^k P(\pi')$ and $Q(\pi)=Q(\pi')$.
}
\end{dfn}

Notice that, $\equiv^k$ is defined so that the same function generates equivalence classes on permutations as on tableaux. In light of (\ref{d P commute}) and (\ref{ID and P commute}), $\equiv^2$ is defined on permutations by the natural action of $d^R_i$ on permutations. That is, $d_i^R(\pi)=d_i(\pi)$ unless $i+1 \in \ID(\pi)\cap\ID(d_i(\pi))$, in which case $d_i^R$ acts trivially.  By considering the action of $d_i^R$ on $\fl(\pi|_{[i-1,i+2]})$, we may express all nontrivial actions as
\begin{equation}\label{flat restricted dual}
x1y4 \quad 
x14y \quad 
x41y \quad
x3y4 \quad
x34y,  
\end{equation}
\noindent
where $d_i^R$ swaps the values $x$ and $y$ in $\fl(\pi|_{[i-1,i+2]})$ while fixing values of $\pi$ not in $[i-1, i+1]$.

\begin{prop}\label{Kj commutes}
The action of $K_j$ on permutations commutes with $d^{R}_i$, $\slink$, and $\slink_{*}$.
\end{prop}
\begin{proof}
To allow us to consider all three cases at once, let $\phi= d^{R}_i$, $\slink$, or $\slink_{*}$. For any $\pi\in S_n$, it suffices to consider the effect of each function in question on $(P(\pi),Q(\pi))$. Recall that $K_j$ acts only on $Q(\pi)$, and this action is independent of $P(\pi)$. Similarly, $\phi$ acts only $P(\pi)$, and this action is independent of $Q(\pi)$. Hence $K_j\circ\phi$ and $\phi\circ K_j$, and both act via the RSK correspondence to yield the unique permutation sent to $(P(\phi(\pi)),Q(K_j(\pi)))$, completing the proof.
\end{proof}

\begin{remark}
\textup{
The action of Knuth equivalence is often introduced in relation to the process of jeu de taquin on skew tableaux. Had we taken this approach, the translated result would state that the action of $\slink, \slink_*$, and $d_i^R$ on standard skew tableaux commutes with jeu de taquin.} 
\end{remark}

\begin{cor}\label{positive perms}
For any $k=0,1,2$, let $\mc\subset S_n$ be a union of equivalence classes of $\equiv^{k}$ such that ${f=\sum_{\pi\in \mc} F_{\ID(\pi)}}$ is a symmetric function. Then  
\[ f=\sum_{\lambda\vdash n} c_\lambda s_\lambda, \]
\noindent
where $c_\lambda = |\{\pi \in \mc \colon \pi \in \SYam(\lambda)\}|$.
\end{cor}
\begin{proof}
Each $\equiv^k$ equivalence class on permutations is sent to an $\equiv^k$ equivalence class on tableaux by the function $P$, while preserving inverse descent sets. Consider the disjoint union of these classes of tableaux. By Theorem~\ref{symmetric to Schur}, $f = \sum c_\lambda s_\lambda$, where $c_\lambda$ is the multiplicity of $U_\lambda$ in this disjoint union. The set $\SYam(\lambda)$ is defined to be the set of permutations sent to $U_\lambda$ by $P$. Thus, $c_\lambda = |\{\pi \in \mc \colon \pi \in \SYam(\lambda)\}|$.
\end{proof}

We end this section by considering generating functions of $\equiv^k$ classes. 

\begin{dfn}\label{fk def}
For $k=0, 1, 2$, let $\{f^{(k)}\}$ be the set of functions that can be realized as $\sum_{T \in \mc } F_{\ID(T)}$, where $\mc$ is a single equivalence class of $\equiv^k$.
\end{dfn}

\noindent
We use the notation $\{f^{(k)}\}$ because it is currently unknown how to best index these sets. In the next section, we will show that each $\{f^{(k)}\}$ forms a spanning set of the quasisymmetric functions.

\begin{conjecture}\label{f2 basis}
The set $\{f^{(2)}\}$ forms a basis for the quasisymmetric functions.
\end{conjecture}

\noindent
This conjecture has been verified up to degree 11 with the aid of a computer. 

\section{Extending to the quasisymmetric Schur functions}\label{quasis} 
We turn our attention to the combinatorics of quasisymmetric Schur functions, placing them near the top of our poset of quasisymmetric functions.

\subsection{preliminaries}

A \emph{standard reverse composition tableau} $T$ is a standard filling of a composition shape that is increasing down the first column, decreasing across rows from left to right, and that satisfies the following additional property.
Given any two values $a>b$ that are adjacent in a row of $T$, no value in $(b,a)$ may occur below and in the same column as $b$. If $a$ is the furthest right value in its row, treat the entry immediately to its right as 0. The set of standard reverse composition tableaux is denoted $\SRCT(\alpha)$. The \emph{bent reading word} of $T\in\SRCT(\alpha)$, $\bw(T)$, is given by reading down columns from right to left and then up the leftmost column. Examples are given in Figure~\ref{tau2}.
We may now the \emph{quasisymmetric Schur functions} by 

\begin{equation}\label{quasi schur def}
\qs_\alpha=\sum_{T\in \SRCT(\alpha)} F_{\ID(\bw(T))}.
\end{equation}

\noindent
Importantly, $\{\qs_\alpha\}$ is a basis for the quasisymmetric functions. For a more thorough treatment of the quasisymmetric Schur functions, see \citep{Luoto2013introduction}.

\begin{remark}
\textup{
Traditionally, the inverse descent set of $T\in\SRCT(\alpha)$ is the set of $i$ such that $i+1$ is in a column weakly to the right of $i$. This is equivalent to defining the inverse descent set via the bent reading word.} 
\end{remark}

A \emph{standard reverse Young tableau} of partition shape $\lambda$ is a standard filling of shape $\lambda$ that is decreasing up columns and across rows from left to right. The set of all standard reverse Young tableau is denoted $\SRT(\lambda)$.
The \emph{reverse column word} of $T\in\SRT(\lambda)$, $\cw(T)$, will be the reverse of the normal convention, proceeding up columns, right to left.
It can then be shown (see \citep[Ch. 3.2]{Luoto2013introduction}) that the basis of Schur functions may be rewritten as
\begin{equation}\label{reverse Schur def}
s_\lambda=\sum_{T\in \SRT(\lambda)} F_{\ID(\cw(T))}.
\end{equation}


The set $\mc(\alpha) \subset \SRT(\lambda(\alpha))$ is given by applying Mason's bijection to $\SRCT(\alpha)$, sorting the columns, and bottom justifying. We denote this map by $\rho$. Importantly, $\rho$ preserves the set of values in each column. Further, it preserves inverse descent sets, since both SRT and SRCT have the property that $i$ is an inverse descent of a filling if and only if $i+1$ is in a column weakly to its right.  Thus,
\begin{equation}
\qs_\alpha=\sum_{T\in \mc(\alpha)} F_{\ID(T)}.
\end{equation}

\noindent
The map $\rho$ also has an inverse defined as follows. Working one column at a time from left to right, start with the second column. Now consider each value $x$ in the column in order from greatest to least.  Out of all remaining locations in the column, place $x$ to the right of highest cell whose value is greater than $x$. Figure~\ref{tau2} gives examples, while the original definition is in \citep{mason2008decomposition}.


\subsection{Transitive actions on $\SRCT(\alpha)$}\label{maximal SRCT} 

For any cell $c$ of $\alpha$, the set of cells weakly below $c$ in its column or weakly above $c$ in the column to its left is called an \emph{$\alpha$-pistol}, as in Figure~\ref{pistols}.


\begin{figure}[h]
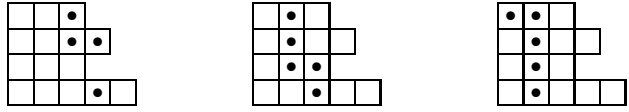
\label{pistols}
   \begin{center} 
   \ytableausetup{smalltableaux}
\begin{ytableau}
{}&{}&\bullet \\
{}&{}&\bullet& \bullet \\
{}&{}&{}  \\
{}&{}&{}&\bullet&{}
\end{ytableau}
\hspace{.5in}
\begin{ytableau}
{}&\bullet&{} \\
{}&\bullet&{}&{} \\
{}&\bullet&\bullet  \\
{}&{}&\bullet&{}&{}
\end{ytableau}
\hspace{.5in}
\begin{ytableau}
\bullet&\bullet&{} \\
{}&\bullet&{}& {} \\
{}&\bullet&{}  \\
{}&\bullet&{}&{}&{}
\end{ytableau}
\hspace{.5in}
%
%
   \end{center}
  \caption{Three $\alpha$-pistols of $\alpha=(5, 3, 4, 3)$.}
 \label{content}
\end{figure}

\noindent
Let 
 $\tilde d_i:S_n \rightarrow S_n$ be the involution that cyclically permutes the values $i-1, i,$ and $i+1$ as
\begin{equation}
\begin{split}
\tilde d_i(\ldots i \ldots i-1 \ldots i+1 \ldots) = (\ldots i-1 \ldots i+1 \ldots i \ldots), \\
\tilde d_i(\ldots i \ldots i+1 \ldots i-1 \ldots) = (\ldots i+1 \ldots i-1 \ldots i \ldots),
\end{split}
\end{equation}
\noindent and that acts as the identity if $i$ occurs between $i-1$ and $i+1$. For example, $\tilde d_3 \circ \tilde d_2(4123) = \tilde d_3(4123) = 3142.$ Let $d_i$ and $\tilde d_i$ act on $\SRCT(\alpha)$ via the bent reading word.

We use $d_i$ and $\tilde d_i$ to define an involution on $\SRCT(\alpha)$. 

\begin{equation} \label{def: D bar}
D_i^{Q}(T) =
 \begin{cases}
  T & \parbox{9cm}{\textup{if $i-1$, $i$, and $i+1$ are contained in two cells of the first column and one in the second column,}}\\
\tilde d_i(T) &  \textup{else, if $i-1$, $i$, and $i+1$ are contained in an $\alpha$-pistol,}\\
 d_i(T) & \textup{otherwise.}
\end{cases}
\end{equation}

\noindent 
 As an example, we may take $\bw(T)=\pi=3265874$ and $\alpha=(3,2,3)$ as in Figure \ref{tau2}. Examples of each of the three cases are as follows. First, $D^Q_7(T)=T$.  Then,  $D^Q_6(T)=\tilde d_6(T)$, which has bent reading word 3257468. Finally, $D^Q_3(T)= d_3(T)$, which has bent reading word 4265873. 

\begin{figure}[H]
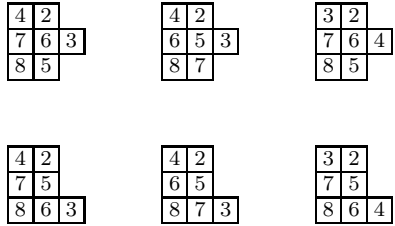

  \ytableausetup{smalltableaux,aligntableaux=bottom}
  \[ 
   \ytableaushort{42, 763,85} \hspace{.4in} 
   \ytableaushort{42, 653, 87} \hspace{.4in} 
   \ytableaushort{32, 764, 85} 
  \] \vspace{.1cm}
    \[ 
   \ytableaushort{42, 75,863} \hspace{.4in} 
   \ytableaushort{42, 65, 873} \hspace{.4in} 
   \ytableaushort{32, 75, 864} 
  \]
  \caption{Three standard reverse composition tableaux of shape $\alpha=(2,3,2)$ above their image via Mason's bijection. From the left, $T\in\SRCT(\alpha)$ with bent reading word $\pi=3265874$, followed by $D^Q_6(T)$, and then $D^Q_3(T)$.}
\label{tau2}
\end{figure}
%

\begin{remark}
\textup{
The function $D_i^Q$ is closely related to its namesake, $D_i$, defined in \citep{assaf2015dual}.
}
\end{remark}

\begin{prop} \label{D transitive}
The set of involutions $D_i^Q$ preserves and acts transitively on $\SRCT(\alpha)$.
\end{prop}
\begin{proof}
The first part, that $D_i^{Q}$ sends $\SRCT(\alpha)$ to itself, is a straight forward matter of considering the possible arrangements of $[i-1, i+1]$, and it is recommended for the reader. We will instead focus on proving the transitivity of the action of $D_i^{Q}$ on $\SRCT(\alpha)$.

The case where $\alpha \vDash1$ is trivial, so we may proceed by induction, letting $\alpha\vDash n$ and assuming the result for compositions of $n-1$. It suffices to show that the value 1 may be moved from the top row to any other possible location of 1. From there, the inductive hypothesis completes the claim. We will consider four cases, as displayed in Figure~\ref{SRCT transitivity}.

First suppose that we wish to move 1 to a column strictly to the left of its position in the top row. Place the 2 in the desired cell and place 3 directly to the left of 1. In this case, $D_i^{Q}$ acts via $d_i$, swapping the 1 with the 2.

Second, suppose we wish to move 1 to a cell below it and in the same column. Place the 3 in the desired cell and place 2 directly to the left of 1. In this case $D_i^{Q}$ acts via $\tilde d_i$, cycling the three numbers.

Next, notice that no cell a single column to the right of 1 is an allowable cell to place 1. This follows by considering the triple formed by 1's original location, the lack of a cell immediately to the right of that, and any desired new location.

Finally, suppose we wish to move 1 at least two columns to the right. Applying the second case, we may move 1 as far down as possible within its own column. Then place the 2 in the desired cell, and place 3 in any allowable location between the 1 and the 2 in the bent reading word. In this fourth case, $D_i^{Q}$ acts via $d_i$, swapping the 1 with the 2. 
\end{proof}

\begin{figure}[H]
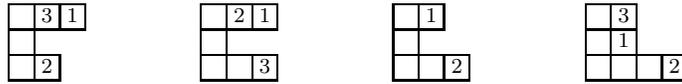

  \ytableausetup{smalltableaux,aligntableaux=bottom}
  \[ \ytableaushort{{}31,{}\none,{}2}  \hspace{.6in}
   \ytableaushort{{}21,{}\none,{}{} 3} \hspace{.6in} 
   \ytableaushort{{}1,{}, {}{}2} \hspace{.6in} 
   \ytableaushort{{}3,{}1,{}{}{}2}
  \]
  \caption{The four cases, in order, demonstrating the transitivity of the action of $D_i^{Q}$ on $\SRCT(\alpha)$.}
\label{SRCT transitivity}
\end{figure}

We want to express the action of $D_i^{Q}$ on $\mc(\alpha)$ via Mason's bijection. Given $T\in\SRT(\lambda)$, define the \emph{elementary quasi-dual equivalence relations} as

\begin{equation} \label{def: d bar}
d^{Q}_i(T) =
 \begin{cases}
 T & \parbox{9cm}{\textup{if $i-1$, $i$, and $i+1$ are in two cells of the first column and one in the second column,}}\\
 d_i(T) & \textup{otherwise.}
\end{cases}
\end{equation}

\noindent
As with earlier involutions, we define the action of $d^{Q}_i$ on a permutation via insertion tableaux.

\begin{prop} \label{p commutes}
For any $T\in\SRCT(\alpha)$,
\[ d^{Q}_i \circ \rho(T)= \rho\circ D_i^{Q}(T)
\]
\end{prop}
\begin{proof}
First notice that each $S\in\SRT(\lambda)$ and each $T\in \SRCT(\alpha)$ are completely determined by the list of values that belong in each column. In the latter case, this is made clear by the algorithm for $\rho^{-1}$. Because $\rho$ preserves the list of values in each column, it suffices to show that $D_i^{Q}$ and $d^{Q}_i$ have the same effect on the values in the columns of $T$ and $\rho(T)$, respectively.

Consider the case where $D_i^Q$ acts trivially on $T$. Since $\rho$ preserves inverse descent sets, $d_i$ and $\tilde d_i$  act trivially on $T$ if and only if they act trivially on $\rho(T)$. Furthermore, $\rho$ preserves the sets of values in each column, so $D_i^{Q}$ acts trivially on $T$ if and only if $d^{Q}_i$ acts trivially on $\rho(T)$.

 Assume that $D_i^Q$ and $d_i^Q$ do not act trivially. In particular, $i-1, i$, and $i+1$ may not all be in the same column, since this would require $i$ to be between $i-1$ and $i+1$ in the reading word. If $i-1, i$, and $i+1$ are all in different columns, then $d^{Q}_i$ and $D_i^{Q}$ will both act by swapping the two outside values, leaving the value in the middle column in place. Thus, it suffices to consider the case where $i-1, i$, and $i+1$ are in exactly two columns.

Without loss of generality, assume  that $i$ shares its column with either $i-1$ or $i+1$ --- if this were not the case, we could consider $D_i^{Q} (T)$ instead. Since $d^{Q}_i$ and $D_i^{Q}$ always change the column of $i$, they must both move $i$ into the same column, forcing $i-1$ and $i+1$ to share the remaining column. Since $D_i^{Q}$ and $d^{Q}_i$ have the same action on the values in each column of $T$ and $\rho(T)$, respectively, we must have  
$d^{Q}_i \circ \rho(T)= \rho\circ D_i^{Q}(T)$.
\end{proof}

\begin{cor}\label{bar d transitive}
The set of $d^{Q}_i$ preserve and act transitively on $\mc(\alpha)$.
\end{cor}
\begin{proof}
The statement follows from Lemma~\ref{D transitive} and Poposition~\ref{p commutes}.
\end{proof}

\begin{remark}
\textup{
In \citep{bessenrodt2011skew}, there is a similar analysis of the relationship between $\rho$ and dual equivalence on SRT for the sake of proving a quasisymmetric Littlewood-Richardson rule. We have made these relationships more explicit by introducing the functions $d_i^Q$ and $D_i^Q$.
}
\end{remark}

Recall from Definition~\ref{fk def} that $\{f^{(k)}\}$ is the set of functions of the form $f=\sum_{T\in C} F_{\ID(T)}$, where $\mc$ is any $\equiv^k$  class. Each $\equiv^k$ also defines relations on $\SRT(\lambda)$ via the reverse column reading word. In this way, it makes sense to compare the equivalence classes generated by $d^Q_i$ to the classes of $\equiv^k$ and  to compare $\{\qs_\alpha\}$ to $\{f^{(k)}\}$.

\begin{thm}\label{bar d coarse}
Quasi-dual equivalence is a coarsening of $\equiv^k$, for each $k=0,1,2$. In particular, each quasisymmetric Schur function is a nonnegative sum of functions in $\{f^{(k)}\}$, for each $k=0,1,2$.
\end{thm}
\begin{proof}
In light of Theorem~\ref{restricted positive}, it suffices to consider $\equiv^{2}$. Notice that $d_i^R$ and $d^{Q}_i$ have the same definition when they act nontrivially on some $T\in \SRT(\lambda)$. It then suffices to show that when $d^{Q}_i$ acts trivially on $T\in\SRT(\lambda)$, so does $d_i^R$. If $d_i$ acts trivially, then the result is immediate, so assume otherwise.
We may thus assume that two values of $[i-1,i+1]$ are in the first column of $T$. Notice that the action of $d_i$ will not change the number of values in $[i-1,i+1]$ that are in the first column and that $i\notin\ID(T)$ implies that $i$ is not in the first column. Hence, if two values of $[i-1, i+1]$ are in the first column of $T$, then $|\ID(T)\cap[i-1,i+1]|=|\ID(d_i(T))\cap[i-1,i+1]|\geq 2$. Comparing with $(\ref{flat restricted dual})$ shows that this never happens when $d_i^R$ acts nontrivially. Hence, $d_i^R$ must act trivially. This completes the proof.
\end{proof}

\begin{remark}
\textup{
The previous proof only used the fact that $d_i^R$ acts trivially when two values in $[i-1, i+1]$ are in the first column, which is less strict than the requirement on $d_i^Q$. If we wished, we could add yet another equivalence relation to our poset by using this intermediate condition. }
\end{remark}

\begin{cor}
The involution $d^{R}_i$ preserves $\mc(\alpha)$.
\end{cor}
\begin{proof}
The corollary is a consequence of Theorem~\ref{bar d coarse}.
\end{proof}

As a side note, the set of $d^{R}_i$ do not always act transitively on $\mc(\alpha)$, as can be seen by inspecting $\alpha=(2,2,2)$.

\begin{cor}
For each $k=0,1,2$, the set of functions $\{f^{(k)}\}$ of the form $f=\sum_{T\in C} F_{\ID(T)}$ forms a spanning set of the quasisymmetric functions.
\end{cor}
\begin{proof}
By Theorem~\ref{bar d coarse}, all quasisymmetric Schur functions must be positive sums of functions in $\{f^{(k)}\}$. The quasisymmetric Schur functions form a basis for the quasisymmetric functions, and so each $\{f^{(k)}\}$ must form a spanning set.
\end{proof}

We also have a new proof of the following, which is stated in \citep[Lemma~2.21]{bessenrodt2014symmetric} as a corollary to the decomposition of Schur functions into quasisymmetric Schur functions in \citep{haglund2011quasisymmetric}.

\begin{cor}
Any symmetric function that is a positive sum of quasisymmetric Schur functions is Schur positive.
\end{cor}
\begin{proof}
This follows from Theorem~\ref{symmetric to Schur} and Theorem~\ref{bar d coarse}.
\end{proof}


%
%
%
%
%
%
%

\section{Extending to shifted dual equivalence}\label{shifted}

\subsection{preliminaries}
A partition $\lambda$ is termed a  \emph{strict partition} if $\lambda =
(\lambda_{1}>\lambda_{2}>\ldots >\lambda_{k}>0)$, the \emph{shifted
  shape} $\lambda$ is the set of squares in positions $\{(i,j)$ given $
1 \leq i\leq k, i\leq j\leq \lambda_{i}+i-1 \}$.  A \emph{standard shifted tableau} $T$ is a standard filling of a shifted shape that is increasing up columns and across rows from left to right. For example, see Figure~\ref{fig: tableaux}.  Let SST$(\lambda)$ be the set of
standard shifted tableaux of shifted shape $\lambda$.

\begin{figure}[h]
 \ytableausetup{aligntableaux=bottom}
\[ 
T=
\begin{ytableau}
\none&\none&7\\
\none&5&6\\
1&2&3&4&8
\end{ytableau},   \hspace{.3in}
h_3(T)=
\begin{ytableau}
\none&\none&7\\
\none&4&6\\
1&2&3&5&8
\end{ytableau}, \hspace{.3in}
 h_5(T)=
\begin{ytableau}
\none&\none&8\\
\none&5&6\\
1&2&3&4&7
\end{ytableau}.
 \]
\caption{Three shifted dual equivalent standard shifted tableau of shape $\lambda=(5,2,1)$.
\label{fig: tableaux}
}
\end{figure}

For a treatment on shifted Knuth equivalence, shifted dual equivalence, and shifted jeu de taquin, see \citep{sagan1987shifted} and \citep{Haiman}. For our purposes, it suffices to define a shifted analog of dual equivalence.

\begin{dfn}[\citep{Haiman}]
Given a permutation $\pi\in S_n$, define the \emph{elementary
shifted dual equivalence} $h_i$ for all $1\leq i\leq n-3$ as follows.
If $n\leq 3$, then $h_1(\pi)=\pi $.
If $n=4$, then $h_1(\pi)$ acts by swapping $x$ and $y$ in the cases below,
\begin{equation}\label{shifted dual}
1x2y \quad x12y \quad 1x4y \quad x14y \quad 4x1y \quad x41y \quad 4x3y \quad x43y,
\end{equation}
and $h_1(\pi) = \pi$ otherwise. 
If $n>4$, then $h_i$ is the involution that fixes values not in $I= [i, i+3]$ and permutes the values in $I$ via $\fl(h_i(\pi)|_I)= h_1(\fl(\pi|_I))$.
\end{dfn}
\noindent
As an example, $h_1(24531)=14532$, $h_2(25134)=24135$, and $h_3(314526)=314526$. Notice that every nontrivial action of $h_i$ is equal to either $d_{i+1}$ or $d_{i+2}$. Two permutations are \emph{shifted dual equivalent} if they are connected by some sequence of elementary shifted dual equivalences.

Given a standard shifted tableau $T$, we define $h_i(T)$ as the result
of letting $h_i$ act on the row reading word of $T$.  Observe that $h_i(T)$ is
also a standard shifted tableau. We can define an
equivalence relation on standard shifted tableaux by saying $T$ and
$h_{i}(T)$ are \emph{shifted dual equivalent} for all $i$. See Figure~\ref{fig: tableaux} for an example.

\begin{thm}[Prop.~2.4, \citep{Haiman}]\label{thm: Haiman}
For all strict partitions $\lambda$, the set of $h_i$ act transitively on $\SSYT(\lambda)$.
\end{thm}

As with $\{f^{(k)}\}$, we wish to consider the generating functions of shifted dual equivalence.

\begin{dfn}\label{fh def}
Let $\{f^{(h)}\}$ be the set of functions that can be realized as $\sum_{T \in \mc } F_{\ID(T)}$, where $\mc$ is a single shifted dual equivalence class of permutations.
\end{dfn}


\noindent
To study $\{f^{(h)}\}$, we will need some intermediary involutions. For an inverse descent set $D\subset [n-1]$, let $D^C$ be the compliment of $D$ in $[n-1]$. The involution on quasisymmetric functions that sends $F_{D}$ to $F_{D^C}$ and $s_\lambda$ to $s_\lambda'$ is denoted $\omega$. Given $\pi\in S_n$, let the \emph{reverse} of $\pi$, $\rev(\pi)$, be the result of reading $\pi$ from right to left. Let the \emph{flip} of $\pi$, $\flip(\pi)$, be the result of taking $\rev(\pi)$ and inverting all values, sending $i$ to $n+1-i$. For example, $\rev(21435)=53412$ and $\flip(21435)=13254$. Notice that $\ID(\rev(\pi))= \ID(\pi)^C$, and $\ID(\flip(\pi))$ is the result of sending each $i\in \ID(\pi)$ to $n-i$. 
Thus, if $f=\sum_\pi F_{\ID(\pi)}$, then
\begin{equation}\label{compliment conjugate}
\sum_\pi F_{\ID(\rev(\pi))}=\sum_\pi F_{\ID(\pi)^C}=\omega(f).
\end{equation}
Further, if $s_\lambda=\sum_{\pi} F_{\ID(\pi)}$, then
 \begin{equation}\label{Schur comp}
 s_{\lambda'}=\sum_{\pi} F_{\ID(\rev(\pi))} \quad \textup{and} \quad s_\lambda =\sum_{\pi} F_{\ID(\flip(\pi))}.
 \end{equation}
  Though (\ref{Schur comp}) is well known, we can sketch a proof as follows. By considering the equation for $s_\lambda$ in (\ref{Schur def}), we may send each $T\in \SYT(\lambda)$ to its conjugate, sending the values in the $i^{th}$ row $\lambda$ to the $i^{th}$ column of $\lambda'$. The resulting filling in $\SYT(\lambda')$ will have the same inverse descent set as $\rev(\rw(T))$. If we instead invert every value in $T$, we may appeal to (\ref{reverse Schur def}) by sending each value $i$ to $n-i$. We can then use the reverse column reading word to give a filling in $\SRT(\lambda)$ with the same inverse descent set as $\flip(\rw(T))$.

\subsection{From restricted dual equivalence to shifted dual equivalence}

\begin{prop}\label{restricted are shifted}
Given distinct $\pi,\pi' \in S_n$, if $d_i^R(\rev(\pi))=\rev(\pi')$, then $h_{i-1}(\pi)=\pi'$. If $d_i^R(\flip(\pi))=\flip(\pi')$, then $h_{n-i}(\pi)=\pi'$.
\end{prop}
\begin{proof}
Taking the reverse of each nontrivial action of $d_i^R$, as expressed in (\ref{flat restricted dual}), gives
\begin{equation}
x14y \quad 
x41y \quad 
4x1y \quad
4x3y \quad
x43y,
\end{equation}
each of which is present in (\ref{shifted dual}). Since $h_i$ acts on $[i, i+3]$ and $d_i^R$ acts on $[i-1, i+2]$, the action of $d_i^R$ is achieved by $h_{i-1}$, completing the first part of the proposition.

Taking the flip of each nontrivial action of $d_i^R$, as expressed in (\ref{flat restricted dual}), gives
\begin{equation}\label{flip word}
x41y \quad 
x14y \quad 
1x4y \quad
1x2y \quad
x12y,
\end{equation}
each of which is present in (\ref{shifted dual}). Taking the flip of the word sends $i$ to $n+1-i$. Again, $h_i$ acts on $[i, i+3]$ while $d_i^R$ acts via the flip of $\pi$ on $[n-i-1, n-i+2]$, so $d_{i}^R$ is achieved by $h_{n-i-1}$. This completes the proof of the second part of the proposition.
\end{proof}


 \begin{thm}\label{fh positive}
 For $k=0,1,2$, the relation $\equiv^k$ on $S_n$ is a refinement of shifted dual equivalence on the reverse of the permutations in $S_n$. In particular, every function in $\{f^{(h)}\}$ is a nonnegative sum of functions in $\{\omega(f)\colon f\in \{f^k\} \}$, for each $k=0,1,2$. Further, the relation $\equiv^k$ on $S_n$ is a refinement of the relation generated by shifted dual equivalence on the flip of the permutations in $S_n$.
 \end{thm}
 \begin{proof}
 By Theorem~\ref{restricted positive}, it suffices to only consider the $k=2$ case. The statements about refining equivalence relations are then direct applications of Proposition~\ref{restricted are shifted} to the definitions of $\equiv^2$ and shifted dual equivalence. The statement about $\{f^{h}\}$ follows directly from Definitions~\ref{fk def} and \ref{fh def}, along with (\ref{compliment conjugate}).
 \end{proof}
 
\begin{cor}\label{fh Schur positive}
Let $\mc$ be any union of shifted dual equivalence classes of $S_n$ such that $f=\sum_{\pi \in \mc} F_{\ID(\pi)}$ is symmetric. Then 
\begin{align*}
f&= \sum_{\lambda\vdash n} |\{\pi \in \mc \colon \rev(\pi) \in \SYam(\lambda)\}| \cdot s_{\lambda'},
\\
&= \sum_{\lambda\vdash n} |\{\pi \in \mc \colon \flip(\pi) \in \SYam(\lambda)\}| \cdot s_{\lambda}.
\end{align*}
In particular, $f$ is Schur positive.
\end{cor}
\begin{proof} 
By Proposition~\ref{restricted are shifted}, reversing all of the permutations in $\mc$ yields some $\md$ that is the union of $\equiv^2$ classes. By Corollary~\ref{positive perms}
\begin{equation}
\sum_{\pi\in\mc} F_{\ID(\rev(\pi))} = \sum_{\lambda\vdash n} |\{\pi \in \md \colon \pi \in \SYam(\lambda)\}| \cdot s_{\lambda}.\\
\end{equation}
Reversing each permutation has the effect of conjugating the resulting Schur functions, so
\begin{equation}
\sum_{\pi\in\mc} F_{\ID(\pi)} = \sum_{\lambda\vdash n} |\{\pi \in \mc \colon \rev(\pi) \in \SYam(\lambda)\}| \cdot s_{\lambda'}.
\end{equation}
Similarly, By Proposition~\ref{restricted are shifted}, flipping all of the permutations in $\mc$ yields some $\md$ that is the union of $\equiv^2$ classes. By Corollary~\ref{positive perms}
\begin{equation}
\sum_{\pi\in\mc} F_{\ID(\flip(\pi))} = \sum_{\lambda\vdash n} |\{\pi \in \md \colon \pi \in \SYam(\lambda)\}| \cdot s_{\lambda}.\\
\end{equation}
Flipping each permutation has no  effect on the resulting the Schur functions, so
\begin{equation}
\sum_{\pi\in\mc} F_{\ID(\pi)} = \sum_{\lambda\vdash n} |\{\pi \in \mc \colon \flip(\pi) \in \SYam(\lambda)\}| \cdot s_{\lambda}.
\end{equation}
\end{proof}
 
 \begin{remark}
 \textup{Traditionally, the generating function over standard shifted tableaux uses peak sets rather than inverse descent sets. Accordingly, it is written as a sum of peak quasisymmetric functions rather than fundamental quasisymmetric functions. It is thus surprising to find a simple condition for Schur positivity with the less studied generating functions of  shifted dual equivalence classes given above.}
 \end{remark}

We end by considering when $\equiv^2$ gives an entire shifted dual equivalence class. Put a different way, our final result gives a converse to Proposition~\ref{restricted are shifted}, so long as we only consider the row reading words of shifted tableaux in $\SSYT(\lambda)$.  
 
 \begin{prop}\label{ssyt are equiv2}
For all strict partitions $\lambda$, the set of $d_i^R$ acts transitively on $\SSYT(\lambda)$ via the flip of the row reading word. 
 \end{prop}
\begin{proof}
Since $d_i^R$ acts via the flip of a row reading word, it acts via the actions in (\ref{flip word}). By Theorem~\ref{thm: Haiman} and Proposition~\ref{restricted are shifted}, this action is an involution on  $\SSYT(\lambda)$. To prove transitivity, we proceed by induction on $n$.

If $n=1$, then the result is trivial. Assume $n\geq 2$ and that the result holds for strict partitions of size less than $n$. Consider any $T\in \SSYT(\lambda)$ such that $n$ is in the northeast corner, $c_1$, in the highest row.  It suffices to show that we may move $n$ into each of the lower northeast corners by applying the involutions in (\ref{flip word}), since we may use the inductive hypothesis to rearrange values in $[n-1]$ as needed. If there are no other northeast corners, we are done.  Otherwise, the inductive hypothesis guarantees that $n-1$ may be placed in the desired corner, $c_2$. Then the values $n-3$ and $n-2$ may be placed in the last of the allowable cells before $c_2$, in row reading order, as in Figure~\ref{shifted swap}. One of the last two involution in (\ref{flip word}) then swaps the $n$ and $n-1$, as desired.
\end{proof}  
 %
 \begin{figure}[h]
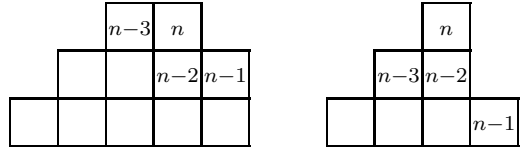

\[
\ytableausetup{smalltableaux=off}
 \begin{ytableau}
\none&\none&\scriptstyle {n-3}&\scriptstyle n\\
\none&&&\scriptstyle {n-2}&\scriptstyle  {n-1}\\
&&&&
\end{ytableau}  
\hspace{.4in}
 \begin{ytableau}
\none&\none&\scriptstyle n\\
\none&\scriptstyle n-3&\scriptstyle n-2\\
&&&\scriptstyle n-1
\end{ytableau}  
\]
 \caption{Example placements of the values in $[n-3,n]$ to move $n$ from the top row to a lower row.} \label{shifted swap}
 \end{figure} 
 %
 %
\bibliographystyle{abbrvnat}
\bibliography{BIBDatabase}

\end{document}